\RequirePackage[l2tabu, orthodox]{nag} %

\documentclass[preprint,12pt]{elsarticle}

\usepackage[margin=1in]{geometry}

\usepackage{amsmath}
\usepackage{amssymb}
\usepackage{amsthm}
\newtheorem{theorem}{Theorem}
\newtheorem{proposition}{Proposition}
\newtheorem{assumption}{Assumption}

\usepackage{graphicx}%
\graphicspath{{img/}} %
\usepackage{dcolumn}%
\usepackage{bm}%
\usepackage{siunitx} %

\usepackage{xcolor} %

\usepackage{hyperref} %
\usepackage{cleveref} %
\crefname{assumption}{assumption}{assumptions}

\usepackage[OT1]{fontenc}
\usepackage[strict=true]{csquotes} %
\usepackage[strings]{underscore} %

\hypersetup{
  pdftitle={Multi-market Oligopoly of Equal Capacity},
  pdfauthor={Ruda Zhang and Roger Ghanem},
  pdfprintscaling=None,
}

\begin{document}

\begin{frontmatter}

\title{Multi-market Oligopoly of Equal Capacity}

\author[usc,samsi]{Ruda Zhang\corref{cor1}}
\cortext[cor1]{Corresponding author.}
\ead{rudazhan@usc.edu}

\author[usc]{Roger Ghanem}
\ead{ghanem@usc.edu}

\affiliation[usc]{
  organization={Department of Civil and Environmental Engineering, University of Southern California},
  city={Los Angeles},
  state={CA},
  postcode={90089},
  country={USA}}
          
\affiliation[samsi]{
  organization={The Statistical and Applied Mathematical Sciences Institute},
  city={Durham},
  state={NC},
  postcode={27703},
  country={USA}}

\begin{abstract}
We consider a variant of Cournot competition,
where multiple firms allocate the same amount of resource across multiple markets.
We prove that the game has a unique pure-strategy Nash equilibrium (NE),
which is symmetric and is characterized by the maximal point of a ``potential function''.
The NE is globally asymptotically stable under the gradient adjustment process,
and is not socially optimal in general.
An application is in transportation, where drivers allocate time over a street network.
\end{abstract}


\end{frontmatter}

\section{Introduction}

Cournot oligopoly is a classic model of economic competition,
see \cite{Vives1999} for a modern game-theoretic review of oligopoly theory.
\citet{Selten1970} introduced multiproduct oligopoly models,
see also \cite{Baumol1982, DeFraja1996}.
A similar term is multimarket oligopoly,
which often involves transportation cost among spatially separated markets.
Besides studying the existence and uniqueness of Nash equilibrium,
it is also important to understand the stability of the equilibrium
under different dynamic adjustment processes.
For a multiproduct oligopoly model, \cite{ZhangAM1996} gave sufficient and necessary conditions
for the locally stability of equilibrium under best-response dynamics.

The discussion is more difficult in the presence of capacity constraints.
\citet{Okuguchi1999} studied the equilibrium and dynamics of
multiproduct oligopoly with capacity constraints.
In particular, for asymmetric firms with compact, convex capacities
such that their payoffs are concave, they showed that Nash equilibrium exists.
If the equilibrium is unique, they provided sufficient conditions
for its global asymptotic stability under a dynamic process,
where players adjust strategies proportionally to their expected marginal profit,
given their expectations of the total output of opponents
(also known as expectations a la Cournot) \citep[Thm 6.4.2, 6.4.3]{Okuguchi1999}.
Their used Lyapunov functions to prove dynamic stability.
\citet{Laye2008} proved the uniqueness of equilibrium
if the demand functions are linear and the cost functions are convex,
because the game is equivalent to a convex optimization problem.
Other forms of capacity constraints have been discussed,
such as transmission capacity constraints \cite{Cunningham2002}
or capacity constraints across multiple time period
\cite{Besanko2004, %
  Ishibashi2008}. %
More papers deal with two-period duopoly using mixed strategy equilibrium
\cite{Berg2012, ChenYH2015}.

In this paper we consider multimarket oligopoly with capacity constraints and symmetric firms.
Unlike \cite{Laye2008}, we allow for concave production functions,
which gives non-increasing demand functions.
A situation where symmetry arises is when the players allocate time across different markets.
Consider a fishery that gives out license for fishing.
All fishermen receive standard device so that, all things being equal,
they have no difference in productivity.
Fish abundance differs spatially at the fishery,
so the fishermen can strategize on their allocation of time at different locations.
Another example is the taxi industry,
where drivers search on a street network for passengers \cite{ZhangRD2020dsp,ZhangRD2020driver}.
For such a multimarket oligopoly, we study its Nash equilibrium of the static game,
its stability under dynamic processes, and its economic efficiency.

The rest of this paper is organized as follows.
In \Cref{sec:game}, we define multi-market oligopoly of equal capacity.
In \Cref{sec:equilibrium}, we prove that the game has a unique, symmetric Nash equilibrium.
In \Cref{sec:stability}, we prove that the NE is globally asymptotically stable
under the gradient adjustment process, and discuss other learning rules.
In \Cref{sec:efficiency}, we discuss when the NE is socially optimal or not,
and conclude with a takeaway.

\section{Game Setup} %
\label{sec:game}

In this section we formalize the game of multi-market competition among firms of equal capacity,
denoted as $\mathcal{G}$.
For the convention of notations,
we use subscript $i$ for an individual firm (or player),
subscript $-i$ for the opponents of firm $i$,
and subscript $x$ to denote a market (or product).
Boldface denotes a vector. Single subscript indicates summation.

Player strategy.
Let $N = \{1, \dots, n\}$ be a set of firms and $E = \{1, \dots, m\}$ be a set of markets.
Each firm $i \in N$ has a unit of resource,
and allocates $s_{ix} \in [0, 1]$ amount of resource in a market $x \in E$
so that $\sum_{x=1}^m s_{ix} = 1$.
The strategy of a firm is the vector of resource distribution $\mathbf{s}_i = (s_{ix})_{x=1}^m$,
which is in its strategy space $S_i = \Delta^{m-1}$.
Here, $\Delta^{m-1} = \{\mathbf{v} \in \mathbb{R}^m : \mathbf{v} \ge 0,
\mathbf{v}^{\text{T}} \mathbf{1} = 1\}$ is the $(m-1)$-dimensional simplex.
Denote strategy profile $\mathbf{S} = (\mathbf{s}_i)_{i=1}^n$,
considered as a vector of length $m \times n$.
Denote strategy space $S = \prod_{i=1}^n S_i$.

Production functions.
The total revenue $u_x$ in a market depends on
the total resource $s_x = \sum_{i=1}^n s_{ix}$ allocated in the market.
Without loss of generality, for each $x \in E$, let
$u_x(s_x): \mathbb{R}_{\ge 0} \mapsto \mathbb{R}_{\ge 0}$ and $u_x(0) = 0$.

Player payoff.
Assume that revenue in each market is distributed proportionally to resource allocation,
and all players have the same cost function $c(\mathbf{v}): \Delta^{m-1} \mapsto \mathbb{R}_{\ge 0}$,
where $\mathbf{v} \in \Delta^{m-1}$.
Because player payoff $u_i$ is the total revenue a player gets from all the markets minus the cost,
it can be written as:
\begin{equation}
  \label{eq:player-payoff}
  u_i(\mathbf{s}_i; \mathbf{s}_{-i}) = \sum_{x=1}^m p_x(s_x) s_{ix} - c(\mathbf{s}_i)
\end{equation}
Here, $p_x: \mathbb{R}_{> 0} \mapsto \mathbb{R}_{\ge 0}$ is the revenue per investment in a market,
defined by $p_x(s_x) = u_x / s_x$;
and $\mathbf{s}_{-i} = \sum_{j \in N}^{j \ne i} \mathbf{s}_j$
is the aggregate strategy of the opponents.
Note that $s_x = s_{ix} + s_{-ix}$.
Denote the aggregate strategy space of the opponents as
$S_{-i} = (n-1) \Delta^{m-1}$, then $\mathbf{s}_{-i} \in S_{-i}$.

Auxiliary functions.
With aggregate strategy $\mathbf{s} = \sum_{i=1}^n \mathbf{s}_i \in n \Delta^{m-1}$,
define marginal player payoff in a market at equilibrium $\phi_x(\mathbf{s})$
and potential function $\Phi(\mathbf{s})$:
\begin{equation}
  \label{eq:equilibrium-marginal-payoff}
  \begin{aligned}
    \phi_x(\mathbf{s}) &= p_x(s_x) + p'_x(s_x) s_x/n - \partial_x c(\mathbf{s}/n) \\
    &= \left(1 - \frac{1}{n}\right) \frac{u_x(s_x)}{s_x} + \frac{1}{n} u'_x(s_x)
    - \partial_x c(\mathbf{s}/n)
  \end{aligned}
\end{equation}
\begin{equation}
  \label{eq:potential}
  \Phi(\mathbf{s}) %
  = \sum_{x=1}^m \left[\left(1 - \frac{1}{n}\right) \int_0^{s_x} \frac{u_x(t)}{t}~\mathrm{d}t
  + \frac{1}{n} u_x(s_x) \right] - n c(\mathbf{s}/n)
\end{equation}
Here, $\partial_x$ is the partial derivative operator with respect to the $x$-th variable.
Note that %
$\nabla \Phi = (\phi_x)_{x=1}^m$.

We impose the following assumption on function forms.

\begin{assumption}
  \label{ass:function}
  For all $x \in E$, $u_x(s_x)$ is concave and differentiable; %
  $c(\mathbf{v})$ is convex and differentiable;
  at least either all but one $u_x(s_x)$ are strictly concave
  or $c(\mathbf{v})$ is strictly convex.
\end{assumption}

Note that this assumption is quite generic.
Because of the law of diminishing returns,
production functions are usually concave, %
and cost functions are usually convex. %
We require differentiability so that $\phi_x(s_x)$ is well-defined.
The strict concavity or convexity requirement is easy to satisfy as well.
For example, \cite{ZhangRD2020driver} shows that the expected revenue as a function of
vehicle supply on a street segment is increasing, strictly concave, and smooth.
Given this assumption, it is easy to see that $p_x(s_x) = u_x / s_x$ is (decreasing) non-increasing.
Notice that we use parentheses for optional conditions that strengthen a statement.
We adopt this convention throughout this paper for a tighter presentation.

Multi-market oligopoly of equal capacity is defined as the game $\mathcal{G} = (N, S, \mathbf{u})$,
where %
payoffs $\mathbf{u} = (u_i)_{i=1}^n$,
such that \Cref{ass:function} hold.
Later we will see that, unlike the Cournot game, it is not an aggregative game \cite{Selten1970}
or even a generalized aggregative game \cite{Jensen2010,Cornes2012}.
It is also not a strictly convex game \cite{Rosen1965} or a potential game \cite{Monderer1996b},
although it could be considered as a Lyapunov game \cite{Clempner2018}.
In the following sections we derive equilibrium and stability results
without referencing these classes of games.

\section{Equilibrium}
\label{sec:equilibrium}

In this section we prove that the game $\mathcal{G}$
has a unique pure-strategy Nash equilibrium (NE),
which is symmetric and is characterized by the unique maximal point of the potential function.
We follow a list of propositions and then give the main result.
Before getting into the details, we point out the keys to the proofs:
convex game guarantees that NE exists;
equal capacity leads to symmetry;
and strictly concave potential function gives a unique solution.

\begin{proposition}
  \label{prop:payoff-concave}
  $u_i(\mathbf{s}_i; \mathbf{s}_{-i})$ is strictly concave,
  $\forall i \in N, \forall \mathbf{s}_{-i} \in S_{-i}$.
\end{proposition}

\begin{proof}
  Because simplex $S_i$ is a convex subset of the non-negative cone $\mathbb{R}_{\ge 0}^{m}$,
  if $u_i(\mathbf{s}_i; \mathbf{s}_{-i})$ is strictly concave on $\mathbb{R}_{\ge 0}^{m}$,
  $\forall i, \forall \mathbf{s}_{-i} \in S_{-i}$,
  then $u_i(\mathbf{s}_i; \mathbf{s}_{-i})$ is also strictly concave on $S_i$,
  $\forall i, \forall \mathbf{s}_{-i} \in S_{-i}$.
  It suffices to prove the former statement without constraints on opponent strategies, that is,
  $u_i(\mathbf{s}_i; \mathbf{s}_{-i})$ is strictly concave on $\mathbb{R}_{\ge 0}^{m}$,
  $\forall i, \forall \mathbf{s}_{-i} \in \mathbb{R}_{\ge 0}^{m}$.
  By \Cref{eq:player-payoff} and \Cref{ass:function},
  it suffices to show that $p_x(s_x) s_{ix}$ is (strictly) concave on $\mathbb{R}_{\ge 0}$,
  $\forall x, \forall i, \forall s_{-ix} \ge 0$.
  To simplify notations, it is equivalent to show that
  $f_x(s; c) \equiv p_x(s + c) s$ is (strictly) concave
  on $\mathbb{R}_{\ge 0}$, $\forall x, \forall c \ge 0$.
  Proof by definition is straightforward but tedious, so we do not include the steps here.
  The key to this proof is that $u_x(s_x)$ is (strictly) concave
  and $p_x(s_x)$ is (decreasing) non-increasing.
  Strict concavity is guaranteed if any $u_x(s_x)$ or $c(\mathbf{v})$ is strictly concave/convex.
  This proves \Cref{prop:payoff-concave}.
\end{proof}

\begin{proposition}
  \label{prop:NE-exist}
  $\mathcal{G}$ has NE, all strict.
\end{proposition}

\begin{proof}
  By \Cref{prop:payoff-concave}, $\mathcal{G}$ is a convex game,
  which is a game where each player has a convex strategy space
  and a concave payoff function %
  for all opponent strategies.
  A convex game has NE if it has a compact strategy space and continuous payoff functions
  \cite{Nikaido1955}.
  Because a finite product of simplices is compact,
  the strategy space $S = \prod_{i=1}^n \Delta^{m-1}$ is compact.
  Because $u_x(s_x)$ and $c(\mathbf{s}_i)$ are differentiable by \Cref{ass:function},
  player payoff $u_i(\mathbf{s}_i; \mathbf{s}_{-i})$ is differentiable and thus continuous.
  Therefore, $\mathcal{G}$ has NE.
  Since $u_i(\mathbf{s}_i; \mathbf{s}_{-i})$ is strictly concave, all NEs are strict. %
  This proves \Cref{prop:NE-exist}.
\end{proof}

\begin{proposition}
  \label{prop:NE-symmetric}
  $\mathcal{G}$ can only have symmetric NE.
\end{proposition}

\begin{proof}
Given a NE $\mathbf{S}^*$, for all player $i$,
the equilibrium strategy $\mathbf{s}_i^*$ solves the following optimization problem,
which is convex by \Cref{prop:payoff-concave}:
\begin{equation}
  \label{eq:optimization}
  \begin{aligned}
    & \text{maximize} && u_i(\mathbf{s}_i; \mathbf{s}_{-i}^*) \\
    & \text{subject to}  && \mathbf{s}_i \ge 0 \\
    &                    && \mathbf{s}_i \cdot \mathbf{1} = 1
  \end{aligned}
\end{equation}
Since this convex optimization problem is strictly feasible,
by Slater's theorem, it has strong duality. 
Since the objective function $u_i(\mathbf{s}_i; \mathbf{s}_{-i}^*)$ is differentiable,
the Karush-Kuhn-Tucker (KKT) theorem states that optimal points of the optimization problem
is the same with the solutions of the KKT conditions:
\begin{equation}
  \label{eq:kkt}
  \begin{aligned}
    &\nabla u_i + \boldsymbol{\lambda}_i - \nu_i \mathbf{1} = 0 && \text{(saddle point)} \\
    &\mathbf{s}_i \ge 0                       &&\text{(primal constraint 1)} \\
    &\boldsymbol{\lambda}_i \ge 0             &&\text{(dual constraint)} \\
    &\mathbf{s}_i \circ \boldsymbol{\lambda}_i = 0     &&\text{(complementary slackness)} \\
    &\mathbf{s}_i \cdot \mathbf{1} = 1        &&\text{(primal constraint 2)}
  \end{aligned}
\end{equation}
Here, operator $\circ$ denotes the Hadamard product: $(x \circ y)_i = x_i y_i$.
Given the saddle point conditions
$\partial u_i / \partial s_{ix} + \lambda_{ix} - \nu_i = 0, \forall x$,
the dual constraint implies that the marginal payoff for player $i$ in market $x$ is bounded above:
$\partial u_i / \partial s_{ix} \le \nu_i, \forall x$.
If player $i$ invests in market $x$, $s_{ix}^* > 0$,
by complementary slackness the upper bound is tight,
$\partial u_i / \partial s_{ix} = \nu_i$,
which means the marginal payoffs for player $i$ are uniform in all markets where $i$ invests.
From \Cref{eq:player-payoff}, marginal player payoffs have the form:
\begin{equation}
  \label{eq:marginal-player-payoff}
  \frac{\partial u_i}{\partial s_{ix}}(\mathbf{s}_i; s_{-ix}) =
  p_x(s_x) + p'_x(s_x) s_{ix} - \partial_x c(\mathbf{s}_i)
\end{equation}
Since $p'_x \le 0$, we have
$p_x(s_x^*) \le \nu_i + |p'_x(s_x^*)| s_{ix}^* + \partial_x c(\mathbf{s}_i), \forall x$,
with equality if $s_{ix}^* > 0$.
If player $i$ invests more in market $x$ than player $j$ does, $s_{ix}^* > s_{jx}^* \ge 0$,
this implies
\begin{equation}
  \label{eq:invest-more}
  \nu_i + |p'_x(s_x^*)| s_{ix}^* + \partial_x c(\mathbf{s}_i^*) \le
  \nu_j + |p'_x(s_x^*)| s_{jx}^* + \partial_x c(\mathbf{s}_j^*)
\end{equation}
But because the players have the same capacity,
player $i$ must have invested less in some market $y$ than player $j$ does:
$s_{jy}^* > s_{iy}^* \ge 0$,
which implies
\begin{equation}
  \label{eq:invest-less}
  \nu_j + |p'_y(s_y^*)| s_{jy}^* + \partial_y c(\mathbf{s}_j^*) \le
  \nu_i + |p'_y(s_y^*)| s_{iy}^* + \partial_y c(\mathbf{s}_i^*)
\end{equation}
Denote $\boldsymbol{\delta} = \mathbf{s}_i^* - \mathbf{s}_j^*$,
$E_+ = \{x \in E : \delta_x \ge 0\}$, and $E_- = \{y \in E : \delta_y < 0\}$.
For all $x \in E_+$, multiply $\delta_x$ to \Cref{eq:invest-more}.
For all $y \in E_-$, multiply $|\delta_y|$ to \Cref{eq:invest-less}.
Sum these inequalities together and note that
$\sum_{x \in E_+} \delta_x = \sum_{y \in E_-} |\delta_y|$, we have:
\begin{equation}
  \label{eq:contradict}
  \sum_{x \in E} |p'_x(s_x^*)| \delta_x^2 +
  \sum_{x \in E} (\partial_x c(\mathbf{s}_i^*) - \partial_x c(\mathbf{s}_j^*)) \delta_x \le 0
\end{equation}
If all but one $u_x(s_x)$ are strictly concave, then all but one $p_x'(s_x)$ are nonzero.
Since at least two entries of $\boldsymbol{\delta}$ are nonzero,
this means that the first term in \Cref{eq:contradict} is positive.
Note that the second term can be written as
$(\mathbf{s}_i^* - \mathbf{s}_j^*) \cdot (\nabla c(\mathbf{s}_i^*) - \nabla c(\mathbf{s}_j^*))$.
If $c(\mathbf{v})$ is strictly convex, then equivalently
$c(\mathbf{v}) - c(\mathbf{w}) - \nabla c(\mathbf{w}) \cdot (\mathbf{v} - \mathbf{w}) > 0$,
$\forall \mathbf{v}, \mathbf{w} \in \Delta^{m-1}, \mathbf{v} \ne \mathbf{w}$,
which means $(\nabla c(\mathbf{v}) - \nabla c(\mathbf{w})) \cdot (\mathbf{v} - \mathbf{w}) > 0$
and therefore the second term is positive.
By \Cref{ass:function}, the left-hand side of \Cref{eq:contradict} is positive,
which is a contradiction.
Therefore all players must have the same strategy in equilibrium.
This proves \Cref{prop:NE-symmetric}.
\end{proof}
  
\begin{proposition}
  \label{prop:Phi-concave}
  $\Phi(\mathbf{s})$ is strictly concave.
\end{proposition}

\begin{proof}
Let $P_x(s_x) = \int_0^{s_x} p_x(t) ~\mathrm{d}t$.
Since $P_x(s_x)$ is a differentiable real function with a convex domain,
it is (strictly) concave if and only if it is globally (strictly) dominated by its linear expansions:
$\forall s_0 > 0, \forall s_x \ge 0, s_x \ne s_0$,
$P_x(s_x) - [P_x(s_0) + p_x(s_0) (s_x - s_0)] = \int_{s_0}^{s_x} (p_x(t) - p_x(s_0))~\mathrm{d}t \le 0$.
This is true because $p_x(s_x)$ is (decreasing) non-increasing.
Because $u_x(s_x)$ is also (strictly) concave, $u_x / n + (1 - 1/n) P_x$ is (strictly) concave.
Therefore, the first term of $\Phi(\mathbf{s})$ in \Cref{eq:potential}
is (strictly) concave on the non-negative cone $\mathbb{R}_{\ge 0}^{m}$.
Because the simplex $n \Delta^{m-1}$ is a convex subset of $\mathbb{R}_{\ge 0}^{m}$,
the first term of $\Phi(\mathbf{s})$ is also (strictly) concave on $n \Delta^{m-1}$.
By \Cref{ass:function}, $\Phi(\mathbf{s})$ is strictly concave.
This proves \Cref{prop:Phi-concave}.
\end{proof}

\begin{theorem}
  \label{thm:NE-unique}
  $\mathcal{G}$ has a unique NE,
  which is symmetric and the aggregate strategy maximizes $\Phi(\mathbf{s})$.
\end{theorem}

\begin{proof}
From \Cref{prop:NE-exist,prop:NE-symmetric}, $\mathcal{G}$ has NE, which are symmetric.
For a symmetric NE $\mathbf{S}^*$ with aggregate strategy $\mathbf{s}^*$,
we have player strategy $\mathbf{s}_i^* = \mathbf{s}^* / n$,
and marginal player payoffs in invested markets are the same for all players:
$\nu_i = \nu, \forall i$.
Note that $\partial u_i / \partial s_{ix} = p_x(s_x^*) + p'_x(s_x^*) s_x^*/n
- \partial_x c(\mathbf{s}^*/n) = \phi_x(\mathbf{s}^*)$.
Summing up \Cref{eq:kkt} for all $i \in N$,
and let $\boldsymbol{\lambda} = \sum_{i=1}^n \boldsymbol{\lambda}_i / n$,
we see that $\mathbf{s}^*$ satisfies:
\begin{equation}
  \label{eq:kkt-new}
  \begin{aligned}
    &\nabla \Phi + \boldsymbol{\lambda} - \nu \mathbf{1} = 0 && \text{(saddle point)} \\
    &\mathbf{s} \ge 0                       &&\text{(primal constraint 1)} \\
    &\boldsymbol{\lambda} \ge 0             &&\text{(dual constraint)} \\
    &\mathbf{s} \circ \boldsymbol{\lambda} = 0     &&\text{(complementary slackness)} \\
    &\mathbf{s} \cdot \mathbf{1} = n        &&\text{(primal constraint 2)}
  \end{aligned}
\end{equation}
Since the potential function $\Phi(\mathbf{s})$ has a compact domain $n \Delta^{m-1}$,
by \Cref{prop:Phi-concave}, it has a unique maximal point $\mathbf{s}^+$.
By the KKT theorem, $\mathbf{s}^+$ is the solution set of the KKT conditions,
which is exactly \Cref{eq:kkt-new}.
Therefore, $\mathbf{s}^* = \mathbf{s}^+$.
This proves \Cref{thm:NE-unique}.
\end{proof}

Embedded in the proof is a procedure to find the NE, which is to optimize $\Phi(\mathbf{s})$.
Here we give better insight into this procedure, for the case where the cost function is separable:
$c(\mathbf{v}) = \sum_{x=1}^m c_x(v_x)$.
Since $u_x(s_x)$ is a uni-variate differentiable (strictly) concave function, 
$u'_x(s_x)$ is (decreasing) non-increasing.
Because $p_x(s_x)$ and $-c_x'(v_x)$ is also (decreasing) non-increasing,
$\phi_x(s_x) = u'_x(s_x) / n + (1 - 1/n) p_x(s_x) - c_x'(s_x/n)$ is (decreasing) non-increasing.
Define inverse function $\phi_x^{-1}: \mathbb{R}_{\ge 0} \to \mathbb{R}_{\ge 0}$,
so that $\phi_x^{-1}(\nu) = 0$ for $\nu > \phi_x^{-1}(0)$.
The function is non-increasing, and decreasing for $\nu \le \phi_x^{-1}(0)$.
Then the equilibrium satisfies:
\begin{equation}
  \label{eq:equilibrium}
  s_x^* = \phi_x^{-1}(\nu),\quad\forall x \in E
\end{equation}
Since total investment equals the number of players, $\sum_{x=1}^m s_x = n$,
marginal player payoff in invested markets $\nu$ is determined by:
\begin{equation}
  \label{eq:marginal-bound}
  \sum_x \phi_x^{-1}(\nu) = n
\end{equation}
Because the left-hand side of \Cref{eq:marginal-bound} is decreasing
for $\nu \le \max_x \phi_x^{-1}(0)$ where the left-hand side is positive,
the equation gives a unique solution $\nu$.
Thus, \Cref{eq:equilibrium} gives a unique $\mathbf{s}^*$.
Figure 1 shows this process graphically.

\begin{figure}[tb]
  \label{fig:determination}
  \centering
  \includegraphics[width=0.7\linewidth]{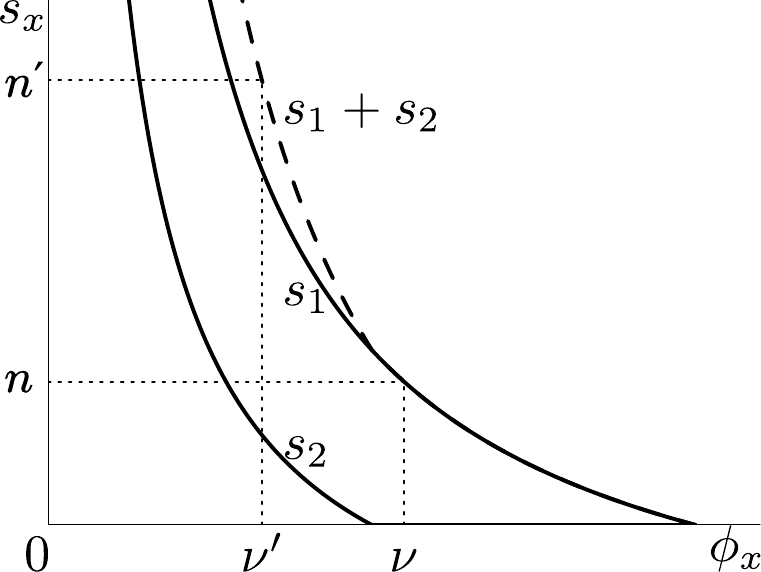}
    \caption{
      Characterization of the Nash equilibrium if cost is separable.
      With $n$ firms, marginal payoff $\nu$ in a market at equilibrium
      is determined by $\sum_{x=1}^m s_x(\nu) = n$.
      Equilibrium allocation in each market can then be determined by $s_x^* = s_x(\nu)$.
  }
\end{figure}

\section{Stability and Learning}
\label{sec:stability}

In this section we prove the global asymptotic stability of the NE under a myopic learning rule
called the gradient adjustment process, and discuss player learning of the NE under other dynamics.
The gradient adjustment process \cite{Arrow1960} is a heuristic learning rule
where players adjust their strategies according to the local gradient of their payoff functions,
projected onto the tangent cone of player strategy space.
Formally, gradient adjustment process is a dynamical system:
\begin{equation}
  \label{eq:gradient-process}
  \frac{\mathrm{d} \mathbf{s}_i}{\mathrm{d} t} =
  P_{T(\mathbf{s}_i)} \nabla_iu_i(\mathbf{s}_i; \mathbf{s}_{-i}),\quad \forall i
\end{equation}
Here $\nabla_i$ denotes the gradient with respect to player strategy $\mathbf{s}_i$,
$T(\mathbf{s}_i)$ is the tangent cone of player strategy space $S_i$ at point $\mathbf{s}_i$,
and $P$ is the projection operator.
For all interior points of $S_i$, $P_{T(\mathbf{s}_i)}$ is simply the centering matrix:
$M_1 = I - \mathbf{1} \mathbf{1}^\text{T} /m$.

\begin{assumption}
  \label{ass:advanced}
  For all $x \in E$, $u_x(s_x)$ is continuously differentiable;
  $c(\mathbf{v})$ is quadratic.
\end{assumption}

\begin{theorem}
  \label{thm:NE-stable}
  With \Cref{ass:advanced},
  the NE of $\mathcal{G}$ is globally asymptotically stable under the gradient adjustment process.
\end{theorem}

\begin{proof}
  To prove that \Cref{eq:gradient-process} is globally asymptotically stable,
  we show that the following function is a global Lyapunov function of the dynamical system,
  i.e. a function that is positive-definite, continuously differentiable,
  and has negative-definite time derivative:
  \begin{equation}
    \label{eq:Lyapunov}
    V(\mathbf{S}) = \Phi_0(\mathbf{s}) + R(\mathbf{S})
    = \left(\Phi(\mathbf{s}^*) - \Phi(\mathbf{s})\right) + \|\mathbf{S} - \bar{\mathbf{S}}\|^2
  \end{equation}
  Here, $\bar{\mathbf{S}} = (\mathbf{s}/n)_{i=1}^n$ is the symmetrized strategy profile.

  Since $\mathbf{s}^*$ is the unique maximal point of $\Phi(\mathbf{s})$,
  this means $\Phi_0(\mathbf{s}) \ge 0$ with equality only at $\mathbf{s} = \mathbf{s}^*$,
  so $\Phi_0(\mathbf{s})$ is positive-definite.
  We add a regularization term $R(\mathbf{S}) = \|\mathbf{S} - \bar{\mathbf{S}}\|^2$,
  which equals zero when $\mathbf{S} = \bar{\mathbf{S}}$ and is positive otherwise.
  Because $\mathbf{S}^* = \bar{\mathbf{S}}^*$,
  we have $V(\mathbf{S}) \ge 0$, with equality only at $\mathbf{S} = \mathbf{S}^*$.
  Therefore, $V(\mathbf{S})$ is positive-definite.
  
  Note that $\nabla_i \Phi_0(\mathbf{s}) = - \nabla_i \Phi(\mathbf{s})
  = - \nabla \Phi(\mathbf{s}) = (-\phi_x(\mathbf{s}))_{x=1}^m$.
  With \Cref{ass:advanced}, every $\phi_x(\mathbf{s})$ is continuous,
  so $\Phi_0(\mathbf{s})$ is continuously differentiable with respect to $\mathbf{S}$.
  With some derivation, we have $\nabla R(\mathbf{S}) = 2 (\mathbf{S} - \bar{\mathbf{S}})$,
  which is apparently continuous.
  Therefore, $V(\mathbf{S})$ is continuously differentiable.

  To simplify discussion, let $\mathbf{S}$ be an interior point of $S = n \Delta^{m-1}$,
  and we study the dynamics of $\Phi_0(\mathbf{s})$ and $R(\mathbf{S})$ separately.
  With \Cref{eq:gradient-process}, the time derivative of $\Phi_0(\mathbf{s})$ is:
\[    \frac{\mathrm{d} \Phi_0(\mathbf{s})}{\mathrm{d} t}
    = -\nabla\Phi(\mathbf{s}) \cdot \frac{\mathrm{d} \mathbf{s}}{\mathrm{d} t}
    = - \sum_{x=1}^m \sum_{i=1}^n \left( \frac{\partial u_i}{\partial s_{ix}} -
      \frac{1}{m} \sum_{y=1}^m \frac{\partial u_i}{\partial s_{iy}}\right) \phi_x \]
  With \Cref{eq:marginal-player-payoff},
  we have $\sum_{i=1}^n \partial u_i / \partial s_{ix} = n p_x(s_x) + p'_x(s_x) s_x
  - \sum_{i=1}^n \partial_x c(\mathbf{s}_i)$.
  With \Cref{ass:advanced}, the cost function can be written as
  $c(\mathbf{v}) = \mathbf{v}^{\text{T}} A \mathbf{v} / 2$,
  where $A$ is a positive semi-definite matrix.
  Note that any linear term $\mathbf{b}^{\text{T}} \mathbf{v}$ can be normalized
  into the production functions $u_x(s_x)$, which becomes $\tilde u_x(s_x) = u_x(s_x) - b_x s_x$.
  Since $\nabla c(\mathbf{v}) = A \mathbf{v}$ is linear, 
  we have $\sum_{i=1}^n \partial_x c(\mathbf{s}_i) = n \partial_x c(\mathbf{s}/n)$,
  and therefore $\sum_{i=1}^n \partial u_i / \partial s_{ix} = n \phi_x(\mathbf{s})$.
  Now we have:
\[    \frac{\mathrm{d} \Phi_0(\mathbf{s})}{\mathrm{d} t}
    = - n \sum_{x=1}^m \left( \phi_x - \frac{1}{m} \sum_{y=1}^m \phi_y \right) \phi_x
    = - mn \left( \overline{\phi^2} - \overline{\phi}^2 \right) \]
  Here, $\overline{\phi} = \sum_{x=1}^m \phi_x / m$ and $\overline{\phi^2} = \sum_{x=1}^m \phi_x^2 / m$.
  Thus, $\overline{\phi^2} - \overline{\phi}^2 \ge 0$,
  with equality if and only if all $\phi_x$ are equal.
  If this is the case, let $\phi_x(\mathbf{s}) = \phi$, where $\phi$ is a constant.
  Then \Cref{eq:kkt-new} is satisfied where $\nu = \phi$ and $\boldsymbol{\lambda} = 0$.
  Therefore $\mathbf{s} = \mathbf{s}^*$.
  So $\mathrm{d} \Phi_0/\mathrm{d} t \le 0$, with equality only at $\mathbf{s} = \mathbf{s}^*$.

  Similarly, the time derivative of $R(\mathbf{S})$ is:
\[    \frac{\mathrm{d} R(\mathbf{S})}{\mathrm{d} t}
    = 2(\mathbf{S} - \bar{\mathbf{S}}) \cdot \frac{\mathrm{d} \mathbf{s}}{\mathrm{d} t}
    = 2 \sum_{i=1}^n \sum_{x=1}^m (s_{ix} - s_x / n) \left( \frac{\partial u_i}{\partial s_{ix}}
      - \frac{1}{m} \sum_{y=1}^m \frac{\partial u_i}{\partial s_{iy}}\right) \]
  Note that $\sum_{y=1}^m \partial u_i / \partial s_{iy}$ does not depend on $x$,
  and $\sum_{x=1}^m (s_{ix} - s_x / n) = 0$, so the last term in the above equation can be dropped.
  With \Cref{eq:marginal-player-payoff}, we have:
\[    \frac{\mathrm{d} R(\mathbf{S})}{\mathrm{d} t}
    = 2 \sum_{i=1}^n \sum_{x=1}^m (s_{ix} - s_x / n) \left(p_x(s_x) + p'_x(s_x) s_{ix}
      - \partial_x c(\mathbf{s}_i) \right)
\]  %
  Note that $p_x(s_x)$ does not depend on $i$, and $\sum_{i=1}^n (s_{ix} - s_x / n) = 0$,
  so $p_x(s_x)$ in the above equation can be dropped.
  The term with $p'_x(s_x) s_{ix}$ expands to
  $2 \sum_{x=1}^m p'_x(s_x) (\sum_{i=1}^n s_{ix}^2 - s_x^2 / n)$,
  which is non-positive because $p'_x(s_x) \le 0$,
  and it equals zero when $\mathbf{s}_i = \mathbf{s}/n$ for all $i$.
  The term with $- \partial_x c(\mathbf{s}_i)$ simplifies to
  $2\sum_{i=1}^n (\mathbf{s}_i - \mathbf{s}/n)^{\text{T}} (-\nabla c(\mathbf{s}_i))$.
  Note that $\nabla c(\mathbf{s}_i) = A \mathbf{s}_i$, this is equal to
  $-2\sum_{i=1}^n (\mathbf{s}_i - \mathbf{s}/n)^{\text{T}} A (\mathbf{s}_i - \mathbf{s}/n)$,
  which is non-positive because $A \ge 0$,
  and it equals zero when $\mathbf{s}_i = \mathbf{s}/n$ for all $i$.
  With \Cref{ass:function}, $\mathrm{d} R / \mathrm{d} t$
  is non-positive and equals zero only when $\mathbf{s}_i = \mathbf{s}/n$ for all $i$,
  that is, $\mathbf{S} = \bar{\mathbf{S}}$.
  Together with the result on $\mathrm{d} \Phi_0/\mathrm{d} t$,
  we have $\mathrm{d} V / \mathrm{d} t \le 0$, with equality only at $\mathbf{S} = \mathbf{S}^*$,
  so $\mathrm{d} V / \mathrm{d} t$ is negative-definite in the interior of $S$.
  
  If $\mathbf{S}$ is a boundary point of $S$,
  we can drop those $x$ where $s_x = 0$.
  With the same argument, we have $\mathrm{d} V/\mathrm{d} t \le 0$,
  with equality only at $\mathbf{S}^*$.
  Therefore, the time derivative $\mathrm{d} V/\mathrm{d} t$ is negative-definite.
  We have thus shown that $V(\mathbf{S})$ is a global Lyapunov function of the dynamical system,
  which immediately implies \Cref{thm:NE-stable}.
\end{proof}

From the proof we see that, during the gradient adjustment process,
the potential function $\Phi(\mathbf{s})$ is increasing as long as
the aggregate strategy $\mathbf{s}$ differs from that of the NE,
and the regularization term $R(\mathbf{S})$ is decreasing as long as
the strategy profile $\mathbf{S}$ is not symmetric.
In other words, this myopic dynamics simultaneously
pushes the player strategies to be more similar to each other and
pushes the aggregate strategy towards the equilibrium.

The stability result in \Cref{thm:NE-stable} is only intended to show that
under a simple and plausible learning rule,
global asymptotic stability of Nash equilibrium is possible in $\mathcal{G}$,
so that the equilibrium can be empirically observed
given real-world perturbations to the competition.
The gradient adjustment process adopted in this paper is not meant to be
the exact learning rule used in real life, which is hard to determine.
But compared with Bayesian or best-response learning rules,
it is less demanding on the players
as it only requires instantaneous local information
instead of complete information of the game or long-term memory of the players.
And even if some players adopt alternative, non-economic learning rules,
the stability of the equilibrium may well be preserved.
For example, some players may simply choose imitative learning~\cite{Roth1995, Fudenberg2009};
in other words, new players copy the more experienced players,
and less successful players copy the more successful ones.
In this case the Nash equilibrium is still the stable focus as all players adopt the same strategy
and the rational payoff-improving players adjust to the equilibrium.
By imitative learning, newcomers save the possibly long process of strategy adjustment
and quickly converge to the equilibrium strategy.
This allows the equilibrium remain stable under an evolving set of players.

We note that the payoff in $\mathcal{G}$ is not diagonally strictly concave,
which is much stronger than \Cref{prop:payoff-concave},
so the uniqueness and stability results cannot follow \cite{Rosen1965}.
But the eigenvalues of the Jacobian $\nabla \left({\mathrm{d} \mathbf{s}}/{\mathrm{d} t}\right)$
are always negative, so local asymptotic stability at the equilibrium is guaranteed
under gradient dynamics with individual-specific adjustment speeds.
We also note that unlike the Cournot game,
$\mathcal{G}$ is not an aggregative game defined in \cite{Selten1970}
or later generalizations \cite{Jensen2010,Cornes2012},
because player strategies are multi-dimensional.
Thus it does not inherit the stability under discrete-time best-response dynamics.
It is also not a potential game and thus does not inherit
the general dynamic stability properties in \cite{Monderer1996b}.
Instead, we provided a ``potential function'' that is a global Lyapunov function
for the gradient dynamics.

\section{The Problem of Social Cost}
\label{sec:efficiency}

In this section we show that
the Nash equilibrium of $\mathcal{G}$ is not socially optimal in general,
and conclude with a takeaway message and hint at a possible fix.
A socially optimal strategy profile maximizes the total income
$u \equiv \sum_{i=1}^n u_i = \sum_{x=1}^m u_x - \sum_{i=1}^n c(\mathbf{s}_i)$,
while per \Cref{thm:NE-unique} the equilibrium maximizes the potential function $\Phi(\mathbf{s})$.
Since $u(\mathbf{S}) \ne \Phi(\mathbf{s})$,
total income is generally not maximized at the NE, thus not socially optimal.
In fact, if total income is maximized,
marginal income $u_x' - \partial_x c(\mathbf{s}_i)$
should be the same for all invested markets of all players.
Instead, at NE, $\phi_x$ is the same for all invested markets.
By \Cref{eq:equilibrium-marginal-payoff},
$\phi_x$ includes the marginal cost at equilibrium
and a weighted average of marginal and average revenue,
with more weight on average revenue as the number of players increases.
This difference implies an inefficiency of the NE in general.

However, the NE can be socially optimal given special forms of $u_x(s_x)$.
For example, if market revenues are power functions of the same order:
$u_x(s_x) = a_x s_x^p, a_x > 0, p \in (0, 1)$,
then total payoff $u = \sum_{x=1}^m a_x s_x^p - \sum_{i=1}^n c(\mathbf{s}_i)$ and
player payoff $u_i = \sum_{x=1}^m a_x s_x^{p-1}s_{ix} - c(\mathbf{s}_i)$.
At social optimum, $u_x' = a_x p s_x^{p-1}$
is a constant for all markets $x \in E$.
Since $\sum_{x=1}^m s_x = n$,
social optimal strategy is $s^\dagger_x = n a_x^{1/(1-p)} / \sum_{y=1}^m a_y^{1/(1-p)}, \forall x \in E$.
Let $u_{ix} = u_x s_{ix} / s_x$. At Nash equilibrium,
${\partial u_{ix}}/{\partial s_{ix}} = a_x \left( (p-1) s_x^{p-2} s_{ix} + s_x^{p-1} \right)$
is a constant for all markets and all players.
This means $\sum_{i=1}^m {\partial u_{ix}}/{\partial s_{ix}} = (p-1+n) a_x s_x^{p-1}$
is a constant for all markets, which gives an aggregate strategy
where $s_x^* = n a_x^{1/(1-p)} / \sum_{y=1}^m a_y^{1/(1-p)}$.
Note that $\mathbf{s}^\dagger = \mathbf{s}^*$,
so the NE is socially optimal in this case.

This phenomenon of difference between cooperative and competitive decisions
has been studied for a long time under different names.
Economic inefficiency \cite{Mill1859,Sidgwick1883}
refers to a situation where total income, or social wealth, is not maximized.
The problem of social cost \cite{Pigou1920,Knight1924}
is the divergence between private and social costs or value.
Later developments include external effect \cite{Meade1952}, rent dissipation \cite{Gordon1954},
market failure \cite{Bator1958}, transaction cost \cite{Coase1960},
and price of anarchy~\cite{Koutsoupias1999, Roughgarden2002}.
\cite{Coase1960} dismissed the discussion of social cost,
calling the reduced social income as transaction cost.
Externality evolved out of external effect,
whose proponents typically call for government regulation, see \cite{Davis1962}.
Institution cost is a generalization of transaction cost, see \cite{Cheung1998}.
Algorithmic game theory uses price of anarchy~\cite{Koutsoupias1999, Roughgarden2002}
and price of stability for this inefficiency of equilibria.

Despite the various terminology, the essence of the problem is the same:
when individuals do not have incentive to maximize the total income,
equilibrium naturally will differ from the optimum set,
which by definition results in less total revenue.
Here we propose the main takeaway in the case of multi-market oligopoly of equal capacity:
if a property is heterogeneous in productivity,
the owner cannot obtain the optimal rent by leasing to multiple tenants
without contracting on their allocation of effort.

\section*{Acknowledgments}

The authors thank Ketan Savla, Sami F. Masri, Juan Carrillo,
Matthew Kahn, Hashem Pesaran, and Geert Ridder of USC for helpful comments.
Research is funded by the National Science Foundation (NSF) Grant 14-524
and, in part, by NSF Grant DMS-1638521.

\bibliographystyle{elsarticle-harv}
\bibliography{game.bib}

\end{document}